\documentclass[11pt,oneside,reqno]{article}

\usepackage{amsmath,amsthm,amssymb,color}
\usepackage{bbm}
\usepackage{mathrsfs}
\usepackage{textcase}
\usepackage{pgfplots}
\usepackage{float}
\usepackage{graphicx}
\usepackage{breakcites}

\usepackage[breaklinks=true]{hyperref}
\hypersetup{
colorlinks=true,
linkcolor=black,
urlcolor=blue,
citecolor=black
}

\usepackage[margin=1in]{geometry}

\numberwithin{equation}{section}
\allowdisplaybreaks[4]

\theoremstyle{plain}
\newtheorem{theorem}{Theorem}[section]
\newtheorem{proposition}[theorem]{Proposition}
\newtheorem{lemma}[theorem]{Lemma}

\newtheorem{conjecture}[theorem]{Conjecture}
\newtheorem{problem}[theorem]{Problem}
\newtheorem{question}[theorem]{Question}

\theoremstyle{definition}

\newtheorem{remark}[theorem]{Remark}

\makeatletter

\newcommand{\eq}{\eqref}

\newcommand{\bigo}{\mathrm{O}}
\newcommand{\lito}{\mathrm{o}}

\def\be#1{\begin{equation*}#1\end{equation*}}
\def\ben#1{\begin{equation}#1\end{equation}}
\def\bes#1{\begin{equation*}\begin{split}#1\end{split}\end{equation*}}
\def\besn#1{\begin{equation}\begin{split}#1\end{split}\end{equation}}
\def\ba#1{\begin{align*}#1\end{align*}}
\def\ban#1{\begin{align}#1\end{align}}
\makeatother

\def\tsfrac#1#2{{\textstyle\frac{#1}{#2}}}
\def\eps{\varepsilon}

\def\mc#1{\mathcal{#1}}

\newcommand{\IE}{\mathbb{E}} 
\newcommand{\IP}{\mathbb{P}} 
\newcommand{\var}{\mathrm{Var}}

\newcommand{\nei}{\mc N}
\newcommand{\gr}{\mc G}
\def\ER{Erd\H{o}s-R\'enyi}
\newcommand{\ps}{\mc{P}}

\newcommand{\Bi}{\mathrm{Bi}}
\newcommand{\Ind}{\mathrm{I}}

\newcommand{\dist}{\mathrm{d}}
\newcommand{\Index}{\mathcal{I}}
\newcommand{\mrec}{M_{\mathrm{rec}}}

\newcommand{\Nr}{N_{n,d,r}}
\newcommand{\sph}{\mathcal{S}}
\newcommand{\bz}{{\mathbf Z}}
\newcommand{\mct}{\mathcal{T}}

\begin{document}

\title{\sc\bf\large\MakeUppercase{
Shotgun assembly of labeled graphs}}
\author{
	Elchanan Mossel
	\thanks{University of California, Berkeley and University of Pennsylvania; \texttt{elmos@mit.edu}}
	\and
	Nathan Ross
	\thanks{University of Melbourne; \texttt{nathan.ross@unimelb.edu.au}}
}
\date{\today}
\maketitle

\begin{abstract}
We consider the problem of reconstructing graphs or labeled graphs from neighborhoods of a given radius $r$. 
Special instances of this problem include the well-known:
 DNA shotgun assembly; the lesser-known:
  neural network reconstruction; and a new problem:
  assembling random jigsaw puzzles. 
We provide some necessary and some sufficient conditions for correct recovery 
both in combinatorial terms and for some generative models including random labelings of lattices, \ER\ random graphs, 
and a random jigsaw puzzle model. Many open problems and conjectures are provided.
\end{abstract}

\section{Introduction}
In  this paper we study the problem of inferring a graph with labels from a collection of local ``$r$-neighborhoods" of the graph. 
In particular we ask how large $r$ must be to ensure that a given randomly generated graph with labels can be uniquely identified---up 
to some natural family of isomorphisms which we always take to be rotations in planar graphs---by its $r$-neighborhoods. Note that if the neighborhoods are too small then identifiability may be impossible: if $r=1$ and all of the vertex
labels are the same, then a graph is only identifiable from its usual $1$-neighborhoods if the degree sequence determines a unique graph.
As far as we know graph shotgun assembly for generative models has not been considered before in the level of generality considered here. Some motivating examples include: 
\begin{itemize} 
\item
DNA shotgun assembly: the goal is to reconstruct a DNA sequence from ``shotgunned" stretches of the sequence. 
The theoretical version of this problem is graph shotgun assembly of a path graph with each vertex corresponding to a site in the genome,
 and so is labeled with an $A, C, G,$ or $T$ 
standing for the nucleotides making up DNA. 
The neighborhoods are paths of adjacent vertices of length $r$, which are referred to as ``reads". 
Shotgun assembly is one of the major techniques for reading DNA sequences and so the theoretical problem
is already well understood. A main question is to determine  how large $r$ has to be to reconstruct the sequence with sufficiently high probability 
under different models of vertex labeling, see  e.g., \cite{Arratia1996}, \cite{Dyer1994}, and \cite{Motahari2013} and references therein.
Note too that in practice the reads are different lengths and have errors. 
\item 
Reconstructing neural networks: recent work in applied neuroscience identifies graph shotgun assembly as an
important problem for reconstructing neural networks;
the goal is to reconstruct a big neural network from subnetworks that are observed in experiments
\cite{Soudry2013}.
\item 
A new problem we call the {\em random jigsaw puzzle problem}. Consider a jigsaw puzzle of size 
$n \times n$ where the border between every two adjacent pieces is drawn uniformly at random using one in $q$ shapes of interfaces
which we call ``jigs." How large should $q$ be so that the puzzle can be recovered uniquely? How can this be done efficiently? 
\end{itemize}

The problem considered here is most closely related to the
 famous \emph{reconstruction conjecture} in combinatorics \cite{Kelly1957} \cite{Harary1974}
which can be stated as follows: a graph $G$ on at least $3$ vertices is uniquely determined by the multi-set of all vertex-deleted subgraphs of $G$. Here a vertex deleted subgraph of $G$ is a graph induced on all the vertices of $G$ but one.
In this paper we are interested in reconstructing graphs with labels from seemingly less information: 
given the graph we assume that we are given all $r$-neighborhoods in the graph. 
While the information is more localized, we make the additional assumption that either the graph structure or the labels are random. As is frequently the case in such settings, randomness makes the problem easier. Indeed, we show that for some
popular random graphs models, reconstruction is possible from relatively small neighborhoods. 

The graph shotgun problem is also related to the graph isomorphism problem \cite{babai1980}. It is a famous open problem to determine the complexity of graph isomorphism. In fact, one may consider a variant of the graph isomorphism problem in our setup: given the neighborhoods of two samples drawn from the same generative model, such that the samples were generated either independently or identically, can it feasibly be determined which method of sample generation was used? Part of the difficulty of the problem in this setup is that it may be required to determine if two neighborhoods are isomorphic or not. While we leave the question of graph isomorphism for randomly generated graphs for future work, we note that some of the techniques used for the classical graph isomorphism
problem are related to our results. In particular our techniques for studying dense random graphs in Section~\ref{secdenseER} 
resemble some of the algorithms suggested for graph isomorphisms for some subclasses of graphs~\cite{Cai1992}. 


We also note that  the question of whether an infinite graph is determined by some collection of its finite subgraphs 
has been studied in the context of unimodular and transitive infinite graphs \cite{Aldous2007} \cite{Frisch2016}.

\subsection{General setup, models, main results}

A (deterministic or random) graph $\gr=\gr_N$  with $N$ vertices and labels (again possibly random)
from a finite set on each vertex or edge is given.
Each vertex $v$ has a ``neighborhood" $\nei_r(v)$ of ``radius" $r$ which could be all of the vertices at distance $r$ or some variation (see
the examples below); we assume that location of vertex $v$ is given in $\nei_r(v)$. 
\begin{itemize}
\item[Q1.] (Identifiability) Given each of the $N$ neighborhoods $\nei_r(v)$ for $v$ a vertex in the network, can
we correctly identify (up to a natural family of isomorphisms) the graph $\gr$ and its labels? We view this question as having two parts:
(a) combinatorial criteria for identifiability (or non-identifiability), and (b) the probability of identifiability under
particular
random generative models.

\item[Q2.] (Reconstruction) Assuming identifiability for a given $\gr_N$ and $r$, for $0<\eps<1$,
what is the minimum number, $\mrec(N, r, \eps)$, of samples (with replacement)  from the collection of neighborhoods
that is necessary to ensure that the chance of correctly reconstructing the
network $\gr$ with labels from the sample is at least $1-\eps$?  
\end{itemize}

Questions Q1(a) and Q2 are discussed in Section~\ref{seccombsamp}, where we 
derive general results about combinatorial criteria for identifiability and 
upper and lower bounds on $\mrec(N, r, \eps)$ based on coupon collecting. 
Notably our conditions for non-identifiability require that the graph is not isomorphic to small perturbations of the graph obtained by replacing a neighborhood with a
non isomorphic neighborhood (thus avoiding the difficulty of the reconstruction conjecture). 
In Sections~\ref{seclattice},~\ref{secER}, and~\ref{secpuzz}, Question~Q1(b) is discussed and the general results of Section~\ref{seccombsamp}
are applied in the following three examples.
Let $\dist(v,w)$ denote the distance between two vertices in a graph.
\begin{enumerate}
\item\label{ex1} $\gr$ is the $d$ dimensional $n$-lattice, here denoted $\bz_n^d$, with i.i.d.\ vertex labels from a probability distribution on $\{1, \ldots, q\}$ 
and the neighborhoods $\nei_r(v)$ are the $(n-r+1)^d$ $r$-cubes with orientation;
here our neighborhoods differ slightly from the general setup and $N=\Nr:=(n-r+1)^d$.
The goal is to identify $\gr$ from these neighborhoods.

\item\label{ex2} $\gr$ is an \ER\ random graph with vertex set $V$ of size  $N$ and edge probability $p_N$ where the vertices have no labels (or you can think of each having the same label)
and the  $r$-neighborhoods, $\nei_r(v), v\in V$, are the subgraphs induced by the vertices at distance no greater than $r$ from each vertex. 
The goal is to correctly identify $\gr$ up to graph isomorphism from these neighborhoods.
We also consider labeled versions of the model. 

\item\label{ex3} \emph{The random jigsaw puzzle problem}. $\gr$ is the $n\times n$ lattice and we view each vertex as being the center of a puzzle piece
with each of the four edges receiving one of $q$ jigs.
Thus each vertex is labeled with an ordered $4$-tuple (pieces are oriented) of the $q$ possible labels  (jigs), corresponding to edge labels. Note that adjacent vertices have dependent labels and that ``edge" pieces receive a $4$-tuple of jigs and are not
distinguished. The neighborhoods $\nei_0(v)$
are simply the vertices with labels and correspond to the puzzle pieces.
(See Section~\ref{secpuzz} for a more intuitive description of the model.)
The goal is to correctly identify $\gr$ from these neighborhoods.
\end{enumerate}

The main question we address in these examples is what are conditions on $r$ or $q$ as $N\to\infty$ to ensure
identifiability (or non-identifiability)? We now summarize a subset of our findings and open problems.

\smallskip
\noindent{\bf Example~\ref{ex1}: Lattices.}
In Section~\ref{seclattice}, we find that if the vertices of the lattice are labeled uniformly and independently then, up to constants, 
the asymptotic threshold of $r$  for identifiability is $\log(n)^{1/d}$.
\begin{theorem}\label{thmlatticethreshintro}
For $\bz_n^d$ with vertex labels i.i.d.\ uniform from fixed $q$ labels and taking limits as $n\to\infty$, if for some $\eps > 0$, 
\[
 r^d \leq (1-\eps) \frac{d}{2^{d-1}} \frac{\log n}{\log q},
 \]
then the probability of identifiability from $r$-neighborhoods tends to zero, and
if for some $\eps>0$,
\[
r^d \geq  (1+\eps) 2d   \frac{\log n}{\log q},
\] 
then the probability of identifiability from $r$-neighborhoods tends to one.
\end{theorem}

We conjecture that:
\begin{conjecture} \label{conj:box_week} 
There exists a constant $c_{d,q}$ such that for every $\eps > 0$, when $r^d \geq (1+\eps) c_{d,q} \log n$, the probability of identifiability goes to $1$  as $n\to\infty$, while when $r^d \leq (1- \eps) c_{d,q}\log n$, the probability of identifiability goes to $0$. 
\end{conjecture}

More ambitiously we can ask:

\begin{question} \label{conj:box_threshold} 
Does there exist a constant $c_d$ such that for every $\eps > 0$, when $r^d \geq (1+\eps) c_d \frac{\log n}{\log q}$, the probability of identifiability goes to $1$ as $n\to\infty$, while when $r^d \leq (1- \eps) c_d \frac{\log n}{\log q}$,
the probability of identifiability goes to $0$? 
\end{question}

In both cases finding the value of  the constant, $c_{d,q}$ or $c_d$, is  a challenging open problem. The case of non-uniform labels is also discussed in Section~\ref{seclattice}.

\smallskip
\noindent{\bf Example~\ref{ex2}: \ER\ graphs.}
The results of Section~\ref{secER} show that for $\lambda\not=1$, the asymptotic threshold for identifiability in the sparse \ER\ random graph is $\log(N)$ (up to constants).

\begin{theorem}\label{thmerintro}
For the \ER\ graph on $N$ vertices with $p_N=\lambda/N$ for fixed $\lambda>0$ and taking limits as $N\to\infty$, if for some $\eps>0$
\be{
\frac{r}{\log(N)}<\frac{1}{2(\lambda-\log(\lambda))}-\eps,
}
then the probability of identifiability from $r$-neighborhoods tends to zero.
\begin{itemize}
\item If $\lambda<1$ and for some $\eps>0$, 
\be{
\frac{r}{\log(N)}>\frac{1}{\log(1/\lambda)}+\eps,
}
then the probability of identifiability from $r$-neighborhoods tends to one. 
\item If $\lambda>1$ and $\lambda_*<1$ is the unique solution to $\lambda e^{-\lambda}=\lambda_* e^{-\lambda_*}$, 
and for some $\eps>0$,
\be{
\frac{r}{\log(N)}>\frac{1}{\log(\lambda)}+\frac{2}{\log(1/\lambda_*)}+\eps, 
}
then the probability of identifiability from $r$-neighborhoods tends to one. 
\end{itemize}
\end{theorem}
For $\lambda=1$, the second statement of Theorem~\ref{thmerdiam} below implies that if $rN^{-1/3}\to\infty$, then the probability of identifiability tends to one,
but this is far from the lower bound $\log(N)$ provided by the previous result. 
We make the following conjecture:
\begin{conjecture} \label{conj:sparse_threshold} 
For positive $\lambda\not=1$, there exists a constant $c_\lambda$ such that for every $\eps > 0$,
when $r \geq (1+\eps) c_\lambda \log N$, the probability of identifiability tends to $1$ as $N\to\infty$,
 while when $r \leq (1- \eps) c_\lambda \log N$, the probability of identifiability goes to $0$. 
\end{conjecture}
Natural open problems are to prove the conjecture, find the value of $c_\lambda$, and also to 
better understand the critical case where $\lambda=1$.
The cases of sparse \ER\ with labels and \ER\ with unbounded average degree are also studied in Section~\ref{secER}.
In particular, in the most technical result in the paper we show that if $p_N=\omega(\log(N)^2/ N)$ then neighborhoods of size $3$ are enough to 
ensure identifiability:
\begin{theorem}\label{thmdenseintro}
If $\gr$ is the \ER\ random graph with $N$ vertices and edge probability $p_N$ satisfying $ N p_N / \log(N)^2\to \infty$ as $N \to \infty$ 
and we are given $\nei_3(v)$ for each vertex $v$ in $\gr$, 
then 
the probability of identifiability tends to one.
\end{theorem}

\noindent{\bf Example~\ref{ex3}: Jigsaw puzzle.}
In Propositions~\ref{propjig1} and~\ref{propjig2}, we show that if $q=\lito(n^{2/3})$, then the probability of identifiability tends
to zero and if $q=\omega(n^{2})$, then the probability of identifiability tends to one. 
We do not believe that either the constant $2/3$ or the constant $2$ is sharp but conjecture there is a critical exponent: 
\begin{conjecture} \label{conj:jigsaw} 
For the jigsaw puzzle problem, there exists a constant $c$ such that for all $\eps > 0$ if
\begin{itemize}
\item
$q \leq n^{c-\eps}$ then the probability of identification goes to $0$ as $n \to \infty$ and if 
\item
$q \geq n^{c+\eps}$ then the probability of identification goes to $1$ as $n \to \infty$.
\end{itemize}
\end{conjecture}

A number of additional open problems and conjectures are given in each section and we conclude the paper with a summary of these and other outstanding questions in Section~\ref{secopen}, but mention a few extensions here. 
This work is a first step in a theoretical understanding of the graph shotgun 
assembly problem and so we consider fundamental models of graphs with labels. 
It is of interest to study the problem for more realistic models of random networks, which will require more sophisticated tools. 
We also have only considered either
unlabeled graphs or graphs that have i.i.d.\ labels,
but the questions considered here can naturally be extended to labelings of the graph 
outside of the i.i.d.\ case. For example the graph may be colored by an Ising model or by a uniform proper coloring. Another avenue of future study is 
developing assembly algorithms and analyzing their complexity (our results implicitly use inefficient greedy algorithms). Perhaps good algorithms from related areas such as DNA shotgun assembly may be adapted to our setting. 
Thus the study of graph shotgun 
 assembly raises new problems in random graphs, percolation, Ising/Potts models, as well as algorithmic problems regarding random constraint satisfaction problems and the theory of spin glasses.

Except for the case of dense ER random graphs and the DNA shotgun assembly problem, none of the graph shotgun results we present are tight (meaning the lower and upper bounds match). 
We conclude the introduction with a family of examples for which it is easy to derive tight bounds.

\smallskip
\noindent{\bf The labelled full binary tree.} 
Let $\mct_n$ be the full binary tree with $2^n$ leaves and label each vertex uniformly from the letters $\{1,\ldots, q\}$. We are given the  $1$-neighborhoods
$\nei_1(v)$ of the $2^n-2$ vertices that are not leaves or the root (so we see the labels of the vertex, its two children, and its parent, as well as the genealogical orientation). 
\begin{proposition}
Let $\eps>0$. 
If 
\be{
\frac{\log(q)}{n}< \log(2) - \eps,
}
then the probability of identifiability of the labeled binary tree $\mct_n$ from $1$-neighborhoods tends to zero.
If 
\be{
\frac{\log(q)}{n} > \log(2) + \eps,
} 
then the probability of identifiability of the labeled binary tree $\mct_n$ from $1$-neighborhoods tends to one.
\end{proposition}
\begin{proof}
To prove the first assertion, note that if there are two vertex disjoint edges between levels $n-2$ and $n-1$ of the tree having endpoints with
identical labels, then with positive probability reconstruction is impossible since we can
switch the cherries below these edges (which have different labels with positive probability) 
and obtain a distinct labeling of the tree with the same neighborhoods; c.f., Lemma~\ref{lemblock} along with discussion around Figures~\ref{figerblock1} and~\ref{figerblock2}. Thus we lower bound the probability of 
this event using the second moment method. 

Actually it's enough to consider a set $\mathcal{N}^{\textrm{o}}$ of $2^{n-2}$ neighborhoods of vertices at level $n-1$ that are vertex-disjoint. 
Now, let $B=B_{n,q}=\sum_{\alpha\not=\beta} X_{\alpha,\beta}$, where the sum is over all pairs of neighborhoods $(\alpha, \beta)$ with $\alpha,\beta\in\mathcal{N}^{\textrm{o}}$,
and $X_{\alpha, \beta}$ is the indicator that for the central vertices of  $\alpha$ and $\beta$ have the same label,  the parent vertices have the same label (possibly different from the central vertices), and the two pairs of leaves have different labels (as sets). 
We compute
\be{
\IE B \geq  2^{2(n-3)} (1/q)^2 (1-2/q^2).
} 
A key fact used here and below is that 
\emph{if the labels are chosen uniformly}, then
for two pairs of neighborhoods $(\alpha, \beta)\not=(\gamma, \delta)$,
$X_{\alpha, \beta}$ and $X_{\gamma, \delta}$ are \emph{independent}. 
Note that this independence does not hold for general
distributions and neighborhoods, since $X_{\alpha,\beta}=1$
may change the probability that $X_{\alpha, \delta}=1$ for $\delta\not=\beta$.
Thus we find
\be{
\var B = \sum_{\alpha\not=\beta} \var(X_{\alpha, \beta}) \leq \IE B,
}
and the first claim of the proposition follows by the second moment method.

For the second part of the claim, it's clear that if no two edges have the same labels, then we can piece together the tree 
from the neighborhoods
by overlapping distinct edges. The mean of the number of pairs of edges with the same labels is bounded above by
\be{
2^{2n+2}(1/q)^2,
}
which tends to zero under the hypothesis of the second statement of the proposition and
so the result follows.
\end{proof}

\subsection{Follow Up Work} \label{subsec:follow}
Since posting this article to arXiv, a number of groups have made significant progress on 
Conjecture~\ref{conj:jigsaw}. 
\cite{Bordenave2016} and \cite{Nenadov2017} independently show that the puzzle can be \emph{uniquely} assembled (meaning each piece is put in its exact location; this is stronger than identifiability) for $q\geq n^{1+\eps}$ for any $\eps>0$, and \cite{Martinsson2016}, \cite{Martinsson2017} sharpens this result to $q=\omega(n)$,
showing that identifiability is possible for $q\geq (2+\eps) n$, 
and that identifiability is impossible for $q\leq \frac{2}{\sqrt{e}} n$, as well as providing additional properties about the number and type of solutions; a similar but weaker statement was independently shown in \cite{Balister2017}.
In a different line of subsequent work, \cite{Mossel2015} provide tight asymptotic bounds on the radius of identifiability for random regular graphs with fixed degree
as the number of vertices tends to infinity. 

\section{Combinatorial and sampling results}\label{seccombsamp}

We introduce two concepts that can be used to determine identifiability: blocking configurations and uniqueness of
overlaps. 
For concreteness, specialize to the case where for each vertex $v$, $\nei_r(v)$ is the labeled subgraph
induced by the vertices at distance no greater than $r$ from each vertex.

\subsection{Blocking configurations}
A blocking configuration is a neighborhood structure or pattern such that if it appears then identifiability is impossible. 
For a given example, 
there can be a number of different blocking configurations, though that described in
 Lemma~\ref{lemblock} below is most likely in our examples. 
In random models, we use blocking configurations to get upper bounds on the asymptotic neighborhood size to ensure non-identifiability:
if the neighborhoods grow too slowly, then the chance that a blocking configuration appears tends to one and identifiability is impossible
(or the probability is bounded away from zero and so identifiability isn't assured).

For $t>s>0$ and vertex $v$ of a graph $\gr$, define the sphere 
(or shell) $\sph(v;s,t)$ to be the subgraph formed
by removing all isolated vertices from the subgraph of~$\gr$ induced by vertices~$u$ with $s\leq d(u,v) \leq t$.

\begin{lemma}\label{lemblock}
If $\gr$ is such that there is an $r>0$ and 
vertices $v,w$ such that
\begin{itemize}
\item[$(i)$] $\sph(v;1,2r)=\sph(w;1,2r)$,
\item[$(ii)$] $\dist(v,w) > 2r$, and 
\item[$(iii)$] the graph obtained by switching $\nei_1(v)$ and $\nei_1(w)$ in $\gr$ is not isomorphic to $\gr$,
\end{itemize}
then identifiability from $r$-neighborhoods is impossible. 
\end{lemma}

\begin{proof}
We claim that there are at least two non-isomorphic labeled graphs having the same $r$-neighborhoods as $\gr$:
the true one, $\gr$, and one where $\nei_1(v)$ and $\nei_1(w)$ are switched, denoted by $\gr'$. 
Condition $(i)$ ensures that such a switch is possible 
since the number of vertices at distance one connecting to vertices at distance two and their labels agree for
 $v$ and $w$. Condition~$(iii)$ ensures that $\gr$ and $\gr'$ are not isomorphic (and note in particular that this implies $\nei_1(v)\not=\nei_1(w)$). Denote by $\nei_r'$ the $r$-neighborhoods generated by $\gr'$.

We only need to show that $\gr$ and $\gr'$ generate the same $r$-neighborhoods (including multiplicities). 
From $(ii)$, there is no vertex having both $v$ and $w$ in its $\gr$ $r$-neighborhood.
 Thus we can split vertices into two groups:
 those being within distance $r$ of exactly one of $v$ or $w$ in $\gr$, and those having distance greater than $r$
 from both of $v$ and $w$. For any vertex $x$ in the latter group, 
 the differences in switching $\nei_1(v)$ and $\nei_1(w)$
 are not reflected by (potential) neighbors of $v$ and $w$ that 
 are at distance $r$ from $x$ (since the labels and positions of such vertices have to match), and so $\nei_r(x)=\nei_r'(x)$. 
 
 For the group of vertices within distance $r$ (in $\gr$) of one of $v$ or $w$, Condition $(i)$ implies there is an obvious matching of each vertex $x$ that satisfies either
 \begin{itemize}
 \item $2\leq \dist(x,v)\leq r$ (distance in $\gr$) or,
 \item $\dist(x,v)=1$ and $x$ has a neighbor at distance two from $v$,
 \end{itemize}
 to one having the same distance from $w$ and identical label. Moreover, under this matching, $\nei_r(x)=\nei_r'(y)$ and $\nei_r'(x)=\nei_r(y)$. 
 Finally, by $(i)$, for $x=v, w$ or a neighbor of $v$ or $w$ with no neighbors at distance $2$ from $v$ or $w$, $\nei_r(x)=\nei_r'(x)$. 
Thus $\gr$ and $\gr'$ generate the same $r$-neighborhoods.
\end{proof}

\begin{remark}
To get a better sense of the lemma, it may help to take a look at the discussion around Figures~\ref{figerblock1} and~\ref{figerblock2}. We also stress that Condition~$(iii)$ implies that 
$\nei_1(v)\not=\nei_1(w)$.
\end{remark}

\begin{remark}
Condition $(iii)$ seems a bit unnatural and possibly hard to verify. Indeed,
it is difficult to check in situations where the graph $G$ has many symmetries 
since the graph isomorphism problem is computationally difficult.
However, such symmetry is rare in random graphs and so in our applications of the lemma, Condition $(iii)$ is easy to verify. 
We also note that the condition is reasonable to impose given the difficulty of the 
``reconstruction conjecture" that has been open for more than 50 years. 
\end{remark}

\subsection{Uniqueness of overlaps}

The next result formalizes the intuition that if all of the neighborhoods of a certain size are unique, 
then slightly larger neighborhoods
are enough to ensure identifiability. 
In random models, we use uniqueness of overlaps to get lower bounds on the asymptotic neighborhood
size to ensure identifiability. If the neighborhoods grow quickly enough, then the chance that
all neighborhoods of a slightly smaller size are unique tends to one and identifiability is ensured.

\begin{lemma}\label{lemoverlap}
If $\nei_{r-1}(v)\not=\nei_{r-1}(w)$ for all vertices $v\not=w$,
then there is an algorithm for recovering the graph from $r$-neighborhoods.
\end{lemma}
\begin{proof}
We can sequentially build the network by overlapping neighborhoods of radius $r-1$. Start with some $r$-neighborhood $\nei_r(v)$ and note that the $(r-1)$-neighborhood 
of each neighbor of $v$ is contained in $\nei_r(v)$ and these are all unique by assumption. Thus for each vertex $w\not=v$, we 
examine the $(r-1)$-neighborhoods of neighbors of $w$ and overlap any of these matching the $(r-1)$-neighborhoods of neighbors of $v$.
Repeating this process for each neighbor of $v$ and then continuing for the vertices at distance $2, 3,\ldots$ from $v$, it's clear that 
the process terminates when a connected component is recovered. 
\end{proof}

\begin{remark}
The proof of the lemma is simple because we assume we see not only $\nei_r(v)$, but also which vertex in the neighborhood 
is the ``center" (namely, $v$). We do not investigate here how to relax this condition to the situation where the center $v$ is not given. 
\end{remark}

\subsection{Sampling}
In the regime where we have uniqueness of $(r-1)$-neighborhoods, then once all neighborhoods have been sampled, reconstruction is trivial. Bounds on the probability of reconstruction then easily come from understanding the number of samples needed to see all neighborhoods, which is just the  coupon collector problem.  In this short subsection, we spell out the details around this statement.
Let $\mrec(N,r,\eps)$ be the minimum number of samples of the $r$-neighborhoods of a graph $\gr_N$ on $N$ vertices so that
the chance the graph can be reconstructed from the samples is least $1-\eps$. 

\begin{lemma}\label{lemcoupon}
If for some $r$, $\nei_{r-1}(v)\not=\nei_{r-1}(w)$ for all vertices $v\not=w$, then
\be{
\mrec(N,r,\eps) \leq \lceil N\log(N) -N\log\eps \rceil.
}
\end{lemma}
\begin{proof}
The proof of Lemma~\ref{lemoverlap} implies that it's enough to see all of the neighborhoods, possibly in multiplicities,
since then we can build the network by overlapping the $(r-1)$-neighborhoods of neighbors of the sampled vertex.
The bound in the lemma now easily follows from coupon collecting: if $T$ is the number of samples with replacement required to collect
$N$ distinct coupons, then a union bound implies that for integer $M>0$,
\be{
\IP(T> M)\leq N (1-1/N)^{M}\leq N e^{-M/N}.
}
Now setting $M=\lceil N\log(N) -N \log \eps \rceil$, we find
\bes{
\IP(&\mbox{Can't reconstruct with $M$ samples}) \leq \IP(T>M)\leq \eps,
}
and so $\mrec(N,r,\eps) \leq M$.
\end{proof}

Since there is no hope of reconstruction if there is some vertex that doesn't appear in any of the sampled neighborhoods,  we can also use coupon collecting to get a lower bound 
on $\mrec(N,r, \eps)$ in the general case. Let $|\nei_r(v)|$ denote the number of vertices in $\nei_r(v)$.
\begin{lemma}\label{lemcoupon2}
If the positive integer $M$ is such that
\be{
\frac{\left(\sum_{i=1}^N \left(1-\frac{|\nei_r(v_i)|}{N}\right)^M\right)^2}{\sum_{i,j=1}^N \left(1-\frac{|\nei_r(v_i)\cup\nei_r(v_j)|}{N}\right)^M}\geq \eps,
}
then $\mrec(N,r, \eps)\geq M$.
\end{lemma}
\begin{proof}
Let $W_M$ be the number of vertices that have not appeared in some neighborhood in a sample of size $M$. If $W_M>0$, then we
can't reconstruct with $M$ samples and so 
by the second moment method,
\besn{\label{7}
\IP(&\mbox{Can't reconstruct with $M$ samples})  \geq \IP(W_M>0) \geq \frac{(\IE W_M)^2}{\IE W_M^2},
}
and for any $M$ such that the right-most side of~\eq{7} is greater than $\eps$, the chance of reconstruction
is at most $1-\eps$ which implies $M\leq \mrec(N,r,\eps)$. The result now follows by computing
\ba{
\IE W_M &= \sum_{i=1}^N \left(1-\frac{|\nei_r(v_i)|}{N}\right)^M, \\
\IE W_M^2 &= \sum_{i,j=1}^N \left(1-\frac{|\nei_r(v_i)\cup\nei_r(v_j)|}{N}\right)^M. \qedhere
}
\end{proof}

\section{Labeled lattice models}\label{seclattice}

Recall the setting of Example~\ref{ex1}:
$\gr$ is the $d\geq 2$ dimensional $n$-box $\bz_n^d$ with i.i.d.\ vertex labels and 
neighborhoods the $r$-boxes contained in $\bz_n^d$; note that for these neighborhoods the position of
 $v$ can be inferred from the neighborhood since it's in the center (recall there are only $(n-r+1)^d$ neighborhoods rather than $n^d$).
 Our results for i.i.d.\ uniform labeling are different than the general i.i.d.\ case.

\subsection{Uniform labels}
Assume the vertices of $\bz_n^d$ are labeled uniformly from $ q\geq 2$ labels. 
Our first result uses blocking configurations to obtain an upper bound on the growth of $r$ to ensure 
a positive chance of 
non-identifiability. 
\begin{proposition}\label{proplatticeblock}
Given the $r$-neighborhoods of $\bz_n^d$ with vertex labels i.i.d.\ uniform from $q$ labels,
 the following holds as
$n \to \infty$. 
\begin{itemize}
\item if $(n/r)^{2d} q^{-(2r)^d}\to \infty$, then the probability of identifiability tends to zero, and
\item if $\liminf\limits_{n\to\infty} \left[(n/r)^{2d} q^{-(2r)^d}\right]>0$, then the probability of identifiability is strictly less than one.
\end{itemize}
\end{proposition}
\begin{proof}
Let $\Gamma'=\Gamma'_{n,d, 2r-1}$ be the set of non-overlapping 
neighborhoods of the form 
$x + [0,2r-1]^d$ where all of the coordinates of $x$ are $0$ modulo $2r$
and let 
\be{
B=B_{n, d,r,q}=\sum_{\alpha, \beta \in \Gamma'}
X_{\alpha, \beta},
}
where $X_{\alpha, \beta}$ is the indicator of the event that the labels of $\alpha$ and $\beta$ are equal except for the labels of the center vertices which must be different.
If $B>0$, then identifiability is impossible since 
if $X_{\alpha,\beta}>0$ for some $\alpha,\beta\in \Gamma'$, then identifiability is impossible since, similar to Lemma~\ref{lemblock},  
there are at least two ways to construct a consistent layout of neighborhoods, by switching the
labels of the center vertices.
Note further that the probability that there is an isomorphism of the graph excluding these two neighborhoods is at most $2^d \times (1/q)^{n^d/2-1}$ (since there are $2^d$ possible rotations and each site has to match the label of one other site). 
 
To lower bound $\IP(B>0)$, we use the second moment method and compute
$\IE B$ and $\IE B^2$.
Assume that $n\gg r$ (without loss under the hypotheses of the proposition).
It's easy to see that $|\Gamma'|= \Theta((n/2r)^d)$
and $\IE X_{\alpha, \beta}= (1/q)^{(2r-1)^d-1}(1-1/q)$, which imply that 
\ben{\label{latticeblock1}
\IE B \geq \Theta((n/2r)^{2d}) (1/q)^{(2r-1)^d-1}(1-1/q).
}
Now, $B$ is concentrated since the $X_{(\alpha,\beta)}$ are pairwise independent; due to the labels being chosen uniformly. 
Thus
\[
\var(B) =\sum_{\alpha, \beta \in \Gamma'} \var(X_{\alpha, \beta}) \leq \IE B.
\]
The result follows by the second moment method. 
\end{proof}

We can use uniqueness of overlaps as in Lemma~\ref{lemoverlap} to find a regime where asymptotic reconstruction is assured.
\begin{proposition}\label{proplatticeover}
If $n^{2d}q^{-(r-1)^d}\to 0$ as $n\to\infty$, then the probability of identifiability (of $\bz_n^d$ with i.i.d.\ uniform on $q$ vertex
labels) from $r$-neighborhoods tends to one.
\end{proposition}
\begin{proof}
Let $Y:=Y_{n,d,r,q}$ be the number of pairs of different $(r-1)$-neighborhoods that have the same labels and
we show that $\IE Y\to 0$ as $n\to\infty$, from which the result follows from a minor variation of Lemma~\ref{lemoverlap}. 

Denote the set of $(r-1)$-neighborhoods of 
$\bz_n^d$  by $\Gamma=\Gamma_{n,d, r-1}$ and for $\alpha,\beta\in\Gamma$,
let $Y_{(\alpha,\beta)}$ be the indicator that $\alpha$ and $\beta$ have the same labels. 
It's obvious that if $\alpha\cap\beta=\emptyset$ (meaning the two neighborhoods share no vertices) then 
$\IE Y_{(\alpha, \beta)}=q^{-(r-1)^d}$, but \emph{since the labels are uniform}, straightforward considerations (see below) 
show that in fact
\begin{equation} \label{eq:lex}
\IE Y_{(\alpha, \beta)}=q^{-(r-1)^d} 
\end{equation}
for all $\alpha \neq \beta$.
Thus we find
\be{
\IE Y=\sum_{\alpha,\beta\in\Gamma,\alpha\not=\beta} \IE Y_{(\alpha,\beta)}=[(n-r)^{2d}-1]q^{-(r-1)^d}.
}
To prove (\ref{eq:lex}) formally assume WLOG that $(x,y) \to (x,y)-(i,j)$ is an injective map from $\beta$ to $\alpha$ where $i,j \geq 0$ 
and at least one of $i$ and $j$ is non-zero.  Then we can label  $\alpha \cup \beta$ according to lexicographic order where
\begin{itemize}
\item
If a site is in $\alpha \setminus \beta$ then we label it arbitrarily.
\item 
If a site~$(x,y)$ is in $\beta$ then we label it by looking at the site
$(x,y)-(i,j)$ which was already labeled.
\end{itemize} 

This defines all labelings of $\alpha \cup \beta$ where $\alpha$ and $\beta$ have the same label so the number of such labelings is 
$q^{|\alpha \setminus \beta|}$ while the total number of labelings of $\alpha \cup \beta$  is 
$q^{|\alpha \cup \beta|}$. The proof follows. 
\end{proof}

Theorem~\ref{thmlatticethreshintro} in the introduction is easily established by combining Propositions~\ref{proplatticeblock} and~\ref{proplatticeover}.

\subsection{Non-uniform labels} If the labels are i.i.d.\ but not uniform, we can 
prove a (weaker) analog of Proposition~\ref{proplatticeover}.
Let $p_i$ denote the chance of label $i$ appearing at a site
and $\ps_j=\sum_{i} p_i^j$ denote the probability that $j$ particular sites have the same label.
\begin{proposition}\label{proplatticeovergen}
If $(nr)^{2d}\ps_2^{(r-1)^d}\to 0$ as $n\to\infty$, then the probability of identifiability (of $\bz_n^d$ with i.i.d.\  vertex
labels) from $r$-neighborhoods tends to one.
\end{proposition}
\begin{proof}
As in the proof of Proposition~\ref{proplatticeover}, let $Y$ be the number of $(r-1)$-neighborhoods that have the same 
labels and we show that $\IE Y \to 0$ as $n\to \infty$. 
Similar to the 
proof of Proposition~\ref{proplatticeblock}, denote the set of $(r-1)$-neighborhoods of 
$\bz_n^d$  by $\Gamma=\Gamma_{n,d, r-1}$ and for $\alpha,\beta\in\Gamma$,
let $Y_{(\alpha,\beta)}$ be the indicator that $\alpha$ and $\beta$ have the same labels. 
It's obvious that if $\alpha\cap\beta=\emptyset$ (meaning the two neighborhoods share no vertices) then 
$\IE Y_{(\alpha, \beta)}=\ps_2^{(r-1)^d}$.  If $\alpha\cap\beta\not=\emptyset$, then 
\ben{\label{22}
\IE Y_{(\alpha, \beta)}=\prod_{j\geq 2} \ps_j^{k_j}, 
}
where $j \times k_j$ are the number of sites in the union of $\alpha$ and $\beta$ that need to be matched 
to $j-1$ other sites to ensure $Y_{(\alpha, \beta)}=1$ (c.f., the justification of~\eq{eq:lex} at the end of the proof of Proposition~\ref{proplatticeover}). 
Note that $\sum_{j\geq 2} (j-1) k_j=(r-1)^d$ and that $\sum_{j\geq 2} k_j=|\alpha \cup \beta|-(r-1)^d$,
since this sum is equal to $|\alpha/\beta|$. Using the basic inequality $\ps_j\leq \ps_2^{j/2}$ for $j\geq2$ in~\eq{22}, we find 
\be{
\IE Y_{(\alpha, \beta)} \leq \prod_{j\geq 2} \ps_2^{jk_j/2} = \ps_2^{|\alpha \cup \beta|/2}\leq \ps_2^{(r-1)^d/2};
}
the last inequality is since $|\alpha \cup \beta| \geq (r-1)^d$. Counting the number of overlapping and non-overlapping neighborhoods, we find
\be{
\IE Y \leq n^{2d} \ps_2^{(r-1)^d} +4r^dn^d \ps_2^{(r-1)^d/2}, 
}
from which the result easily follows.
\end{proof}
\begin{remark}
If the labels are uniform, then $\ps_j=q^{-(j-1)}$ and so we can use this exact quantity  (rather than the inequality $\ps_j\leq \ps_2^{j/2}$) in~\eq{22} in the proof of
Proposition~\ref{proplatticeovergen} to recover the sharper Proposition~\ref{proplatticeover}.
\end{remark}
For non-uniform vertex labels, the correlations between the appearance of overlapping blocking sets 
can become significant and so the second moment method of Proposition~\ref{proplatticeblock} breaks down. 
Still we believe that similar results should hold: 

\begin{conjecture} \label{conj:box_threshold2}
Consider a distribution $\pi$ that is fully supported on $\{1,\ldots, q\}$ and the labeling of $\bz_n^d$ by i.i.d.\ labels from $\pi$. For every dimension $d$, 
there exists a constant $c_d(\pi)$ such that for every $\eps > 0$, when $r^d \geq (1+\eps) c_d(\pi) \log n$, the probability of identifiability tends to one as $n\to\infty$, while when $r^d \leq (1-\eps) c_d(\pi) \log n$, the probability of identifiability goes to $0$. 
\end{conjecture}

We believe that conjecture~\ref{conj:box_threshold2} should also extend to some dependent setups including:
\begin{itemize}
\item 
The uniform distribution of legal vertex colorings of a box with $q \geq 3 d$ colors. We require that $q$ is large to ensure correlation decay of the distribution. Note for example that if $q=2$ and $d \geq 2$, then the problem is degenerate as there are only two possible colorings of the graph. 

\item 
The Ising and Potts models with finite temperature $0 < \beta < \infty$ in the box.
\end{itemize} 

Proving the conjectures and establishing the value of the threshold in these examples are fascinating open problems.

\subsection{Sampling}

If $\bz_n^d$ has uniqueness of $(r-1)$-overlaps (asymptotically assured in the regimes of Propositions~\ref{proplatticeover} and~\ref{proplatticeovergen}),
then the argument of Lemma~\ref{lemcoupon} automatically implies an upper bound of order
$N(\log(N)-\log(\eps))$ (recall $N=\Nr:=(n-r+1)^d$ is the number of neighborhoods) 
on $\mrec(N,\eps,r)$, the minimum number of samples needed
to reconstruct the labels of the lattice with probability at least $1-\eps$. We can also use Lemma~\ref{lemcoupon2} to show that we need 
at least of order (large $N$, small $\eps$) $\frac{N}{r^d}\left( \log(N/r^d)-(\log(\eps)\right)$ samples to reconstruct in any regime.

\begin{proposition}
For $\bz_n^d$ with uniform vertex labels,
\ben{\label{408}
\mrec(N,\eps,r)\geq \frac{\log\left(\tsfrac{1}{\eps}-1\right)-\log\left(\frac{(2r-1)^d}{N}\right)}{-\log\left(1-\frac{r^d}{N}\right)}.
}
\end{proposition}
\begin{proof}
We may use Lemma~\ref{lemcoupon2} with this neighborhood structure since 
its argument only relies on the size (and not the structure) of the neighborhoods. 
First note $|\nei_r(v)|=r^d$ for all $v$ and $|\nei_r(v)\cup\nei_r(w)|=2r^d$ if $\nei_r(v)\cap\nei_r(w)=\emptyset$
and  $|\nei_r(v)\cup\nei_r(w)|\geq r^d$ otherwise. Using these bounds, if $M$ is no greater than the right hand side of~\eq{408}, then 
\ba{
\frac{\left(\sum_{i=1}^N \left(1-\frac{|\nei_r(v_i)|}{N}\right)^M\right)^2}{\sum_{i,j=1}^N \left(1-\frac{|\nei_r(v_i)\cup\nei_r(v_j)|}{N}\right)^M}
&\geq \frac{N^2 \left(1-\frac{r^d}{N}\right)^{2M}}{N^2 \left(1-\frac{2r^d}{N}\right)^M+N (2r-1)^d \left(1-\frac{r^d}{N}\right)^M}\\
&\quad\geq \left[1+ \frac{(2r-1)^d}{N} \left( 1-\frac{r^d}{N}\right)^{-M}\right]^{-1} \geq \eps,
}
and the result follows.
\end{proof}

\section{\ER\ graph}\label{secER}
Assume the setup of Example~\ref{ex2}: $\gr$ is the
\ER\ random graph with $N$ vertices, 
the vertices have no labels (or to fit our setup, all labels are the same)
and for each vertex $v$, we have the $r$-neighborhoods $\nei_r(v)$ which are the subgraphs induced by vertices at distance $\leq r$ from $v$.
This example fits exactly into our general setup and so Lemmas~\ref{lemblock} and~\ref{lemoverlap} can be applied ``out of the box".
As is typical for \ER\ random graphs, 
the results differ if the graph has bounded average degree or not
and so we separate our results accordingly to Sections~\ref{secsparseER} 
and~\ref{secdenseER}. 

\subsection{Bounded average degree \ER}\label{secsparseER}
Let $\gr$ be the \ER\ random graph with $N$ vertices and
edge probability $p_N=\lambda/N$ for some $\lambda>0$.
We use the blocking configuration of Lemma~\ref{lemblock}  to show the following result.
\begin{proposition}\label{properblock}
For the \ER\ graph on $N$ vertices with $p_N=\lambda/N$, using the notation of the previous paragraph, and taking limits as $N\to\infty$, 
\begin{itemize}
\item if $\sqrt{N} \lambda^{r}(1-\lambda/N)^{Nr} \to\infty$, then the probability of identifiability tends to zero, and
\item if $\liminf\limits_{N\to\infty}\sqrt{N} \lambda^{r}(1-\lambda/N)^{Nr} >0$, then the probability of identifiability is strictly less than one.
\end{itemize}
\end{proposition}
\begin{proof}
Note that $\lambda (1-\lambda/N)^N<\lambda/(1+\lambda)$ and so if
$r$ grows faster than $\log(N)$, then neither of the hypotheses of the proposition are satisfied, and so we can assume
without loss that $r/N^a\to 0$ for all $a>0$.
For a collection of vertices~$W$ of a graph~$\gr$, we define the \emph{edge-induced} subgraph on $W$ to be the (connected) subgraph formed by all edges of $\gr$ with at least one endpoint in $W$. With this definition, we lower bound the probability of the appearance of the following blocking edge-induced subgraph on $4r+6$ vertices:
the subgraph has two components, one a path graph on $2r+1$ vertices and the other a path graph on $2r+1$ vertices
with the addition of both end vertices being connected to two other vertices with no other edges to form ``prongs"; see Figure~\ref{figerblock1}.

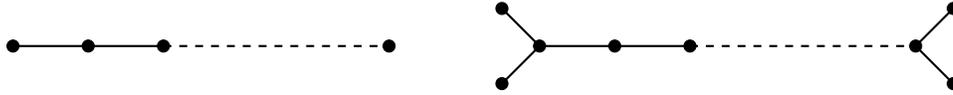
\begin{figure}[H]
\begin{center}
 \begin{tikzpicture}
 \draw [thick] (0,0) -- (2,0);
 \draw [dashed, thick] (2,0) -- (5,0); 
  \draw [thick] (7,0) -- (9,0);
  \draw[thick, dashed] (9,0) -- (12,0);
    \draw [thick] (7,0) -- (6.5,.5);
      \draw [thick] (7,0) -- (6.5,-.5);
      \draw [thick] (12,0) -- (12.5,.5);
      \draw [thick] (12,0) -- (12.5,-.5);
      \draw [fill] (0,0) circle [radius=.08];
        \draw [fill] (1,0) circle [radius=.08];
         \draw [fill] (2,0) circle [radius=.08];
          \draw [fill] (5,0) circle [radius=.08];
           \draw [fill] (7,0) circle [radius=.08];
        \draw [fill] (8,0) circle [radius=.08];
         \draw [fill] (9,0) circle [radius=.08];
          \draw [fill] (12,0) circle [radius=.08];
           \draw [fill] (6.5, .5) circle [radius=.08];
        \draw [fill] (6.5, -.5) circle [radius=.08];
         \draw [fill] (12.5,.5) circle [radius=.08];
          \draw [fill] (12.5,-.5) circle [radius=.08];
 \end{tikzpicture}
 \end{center}
\caption{Example of blocking subgraph for neighborhoods of radius $r$. The path graph has $2r+1$ vertices.}
\label{figerblock1}
\end{figure}

Note that this blocking set satisfies the hypotheses of Lemma~\ref{lemblock} by taking
$v$ to be an endpoint of the path graph and $w$ to be one of the degree three vertices.
Alternatively, it's easy to see that if such a subgraph is present, then identifiability is impossible because 
there are at least two ways to construct the graph consistent with the neighborhoods,
by switching one of the prongs to the path graph; see Figure~\ref{figerblock2} for illustration.

\begin{figure}[H]
\begin{center}
 \begin{tikzpicture}
 \draw [thick] (0,0) -- (2,0);
 \draw [dashed, thick] (2,0) -- (5,0); 
  \draw [thick] (7,0) -- (9,0);
  \draw[thick, dashed] (9,0) -- (12,0);
      \draw [thick] (12,0) -- (12.5,.5);
      \draw [thick] (12,0) -- (12.5,-.5);
        \draw [thick] (5,0) -- (5.5,.5);
      \draw [thick] (5,0) -- (5.5,-.5);
      \draw [fill] (0,0) circle [radius=.08];
        \draw [fill] (1,0) circle [radius=.08];
         \draw [fill] (2,0) circle [radius=.08];
          \draw [fill] (5,0) circle [radius=.08];
           \draw [fill] (7,0) circle [radius=.08];
        \draw [fill] (8,0) circle [radius=.08];
         \draw [fill] (9,0) circle [radius=.08];
          \draw [fill] (12,0) circle [radius=.08];
         \draw [fill] (12.5,.5) circle [radius=.08];
          \draw [fill] (12.5,-.5) circle [radius=.08];
             \draw [fill] (5.5,.5) circle [radius=.08];
          \draw [fill] (5.5,-.5) circle [radius=.08];
 \end{tikzpicture}
\end{center}
\caption{A subgraph that has the same $r$-neighborhoods as that of Figure~\ref{figerblock1}}
\label{figerblock2}
\end{figure}
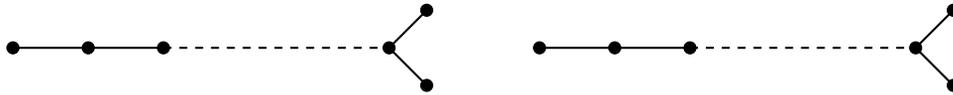

 Let $B=B_{N, r, \lambda}$ be the number of such edge-induced subgraphs of $\gr$ and write
 $B=\sum_{\alpha\in\Gamma} X_\alpha$,
 where $\Gamma=\Gamma_{N,4r+6}$ is the collection of subsets of vertices of size $4r+6$ and for $\alpha\in\Gamma$,
  $X_\alpha$ is the indicator
 that the blocking subgraph of Figure~\ref{figerblock1} is the edge-induced subgraph of $\gr$ on $\alpha$. 
The $X_\alpha$ are equally distributed and for $\alpha\not=\beta$, if $\alpha\cap\beta\not=\emptyset$, 
then $X_\alpha X_\beta=0$. Thus we find for (say) $\alpha=\{1,\ldots, 4r+6\}$ and $\beta=\{4r+7, \ldots, 8r+12\}$,
\bes{
&\IE B = \binom{N}{4r+6} \IE X_\alpha, \hspace{1cm}
  \IE B^2 = \IE B\left(1+\binom{N-4r-6}{4r+6} \IE [ X_\beta|X_\alpha=1]\right).
}

From this point we need to compute $\IE X_\alpha$ and $\IE[X_\beta|X_\alpha=1]$.
There is at most one copy of the blocking edge-induced subgraph on $\alpha$, but
there are a number of ways the subgraph can appear. By enumeration and noting the chance that any potential
way the subgraph can appear, we find
\besn{\label{ermeanblock1}
\IE X_\alpha&=\binom{4r+6}{2r+1}\binom{2r+5}{4}\binom{4}{2}\frac{2 (2r+1)!^2}{2^2} p_N^{2(2r+2)} (1-p_N)^{(4r+6)(N-3)+4};
}
the first binomial coefficient counts the number of ways of assigning $2r+1$ vertices of $\alpha$ to the path graph, the second assigns
four of the remaining vertices to the prongs and for each of the $(2r+1)$-paths, there are $(2r+1)!/2$ ways to put them in order; the final factor of $2$
comes from assigning the pairs of prong vertices to an end. Once the vertices are assigned, there are $4r+4$ edges that must appear,
each with probability $p_N$, and 
$2(2r-1)(N-3)+6(N-2)+2(N-4)$ edges that must not appear.

Similarly, given $X_\alpha=1$, none of the vertices of $\alpha$ have edges connecting to vertices outside of $\alpha$ and so 
$X_\beta|X_\alpha=1$  is distributed as $X_\beta$, but on an \ER\ graph on $N-4r-6$ vertices and 
chance of edge $p_N$. Thus we use~\eq{ermeanblock1} but with $N-4r-6$ replacing $N$ (except in $p_N$) to find
\besn{\label{ermeanblock2}
\IE&[ X_\beta|X_\alpha=1] =\binom{4r+6}{2r+1}\binom{2r+5}{4}\binom{4}{2}\frac{2 (2r+1)!^2}{2^2}  p_N^{2(2r+2)} (1-p_N)^{(4r+6)(N-4r-9)+4}.
}

Putting together~\eq{ermeanblock1} and~\eq{ermeanblock2} and using that under either of the hypotheses of the proposition, $r/N^a\to0$ for any $a>0$, we find
\bes{
&\frac{(\IE B)^2}{\IE B^2}\geq \frac{(N-4r-6)^{4r+6}p_N^{4r+4}(1-p_N)^{N(4r+6)}}{8+(N-4r-6)^{4r+6}p_N^{4r+4}(1-p_N)^{(N-4r)(4r+6)}},
}
and under the first hypothesis of the proposition, the numerator and the denominator tend to infinity at the same rate,
and under the second, the numerator on the right hand side stays bounded away from zero.
\end{proof}

If $r$ is larger than the diameter of the graph, then clearly we can identify from the neighborhoods.
Thus we can use known results on the diameter of the \ER\ random graph (see  \cite{Riordan2010}, \cite{Luczak1998}, \cite{Nachmias2008}, \cite{Addario-Berry2012}) to get a lower bound on the
growth of $r$ to guarantee identifiability.
Denote convergence in probability by $\stackrel{p}{\longrightarrow}$.
\begin{theorem}\label{thmerdiam}
Let $\gr_N$ be the \ER\ random graph on $N$ vertices with edge probability $p_N=\lambda/N$ for a fixed $\lambda>0$ and let
$D=D_{N,\lambda}$ to be the maximum diameter of a component of $\gr_N$.
\begin{itemize}
\item \cite[Theorem~11]{Luczak1998} If $\lambda<1$, then $D_{N,\lambda}/\log(N)\stackrel{p}{\longrightarrow} 1/\log(1/\lambda)$.
\item \cite[Theorem~1.1]{Nachmias2008}, \cite[Theorem~5]{Addario-Berry2012} If $\lambda=1$, then $N^{-1/3}D_{N,1}$ converges in distribution to a non-negative and non-degenerate distribution.
\item \cite[Theorem~1.1]{Riordan2010} If $\lambda>1$, and $\lambda_*<1$ is the unique solution to $\lambda e^{-\lambda}=\lambda_* e^{-\lambda_*}$,
then 
\be{
\frac{D_{N,\lambda}}{\log(N)}\stackrel{p}{\longrightarrow} \frac{1}{\log(\lambda)}+\frac{2}{\log(1/\lambda_*)}.
}
\end{itemize}  
\end{theorem}

Theorem~\ref{thmerintro} in the introduction summarizes the lower bound on the neighborhood size for identifiability given by Proposition~\ref{properblock}
and the upper bounds given by the properties of the diameter of Theorem~\ref{thmerdiam}.

\smallskip
\noindent{\bf Labeled \ER.} 
Assuming 
vertices have i.i.d.\ labels from a finite set and we let $\ps_2$ be the chance that two given
vertices have the same label, we show the following result.
\begin{proposition}
For the labeled \ER\ graph with $p_N=\lambda/N$, using the notation of the previous paragraph, and assuming $\ps_2\not=1$,
if for some $\eps>0$,
\be{
\frac{r}{\log(N)-2\log(1-\ps_2)}< \frac{1}{2\lambda -\log(\lambda ^2 \ps_2)}-\eps, 
}
then the chance of identifiability tends to zero as $N\to\infty$.
\end{proposition}
\begin{proof}
The argument is nearly identical to the proof 
of Proposition~\ref{properblock} but now the blocking configuration 
is two isolated path graphs with $2r+1$ vertices, both having the same labels in the 
$2r-1$ middle vertices, and each having two different labels at the endpoints. Then at least one of the possible switching of labels of two endpoints of each path graph will result in a non-isomorphic labeled graph with the same neighborhoods.
If $B$ is the number of such configurations, then the result follows from the second moment method after computing
\bes{
&\frac{(\IE B)^2}{\IE B^2}\geq \frac{(N-4r-2)^{4r+2}p_N^{4r}(1-p_N)^{4Nr}\ps_2^{2r-1}(1-\ps_2)^2}{
\big((N-4r-2)^{4r+2}p_N^{4r}(1-p_N)^{(N-2r-4)(4r-2)} \ps_2^{2r-1}(1-\ps_2)^2\big)+8}. 
\qedhere}
\end{proof}

We make the following conjecture. 
\begin{conjecture} \label{conj:sparselabel_threshold}
Consider a distribution $\pi$ that is fully supported on $\{1,\ldots, q\}$ and the i.i.d.\ $\pi$-vertex labeling of 
the \ER\ random graph on $N$ vertices with parameter $\lambda/N$. For positive $\lambda\not=1$,  
there exists a constant $c_\lambda(\pi)$ such that for every $\eps > 0$, when $r \geq (1+\eps) c_\lambda(\pi) \log N$, the probability of identifiability tends to one as $N\to\infty$,
 while when $r \leq (1-\eps) c_\lambda(\pi) \log N$, the probability of identifiability tends to $0$. 
\end{conjecture}
Open problems are to establish the conjecture, determine the value of $c_\lambda(\pi)$, and understand 
the critical case where $\lambda=1$.

\subsection{Dense \ER\ graph}\label{secdenseER}

Now we assume that 
$\gr$ is the \ER\ random graph with $N$ vertices and edge probability $p_N$ such that as $N\to\infty$, $ N p_N / \log(N)^2\to \infty$
and the neighborhoods are as before, described in Example~\ref{ex2}. We restate and prove Theorem~\ref{thmdenseintro} from
the introduction.

\begin{theorem}\label{properdense}
If $\gr$ is the \ER\ random graph with $N$ vertices and edge probability $p_N$ satisfying $ N p_N / \log(N)^2\to \infty$ as $N \to \infty$ 
and we are given $\nei_3(v)$ for each vertex $v$ in $\gr$, 
then 
the probability of identifiability tends to one.
\end{theorem}
\begin{proof}
If $p_N>N^{-3/5}$, then, with high probability, the diameter of $\gr$ is at most $3$ \cite{Bollobas1981} and so we can assume without loss that $ p_N\leq N^{-3/5}$.

We show that the chance of the event ``each vertex $v$ has distinct $2$-neighborhood $\nei_2(v)$" tends to one and 
then the result follows by the uniqueness of overlaps Lemma~\ref{lemoverlap}.
If $v$ and $w$ are distinct vertices of $\gr$, then it's enough to show as $N\to\infty$,
\ben{\label{der1}
N^2 \IP\left(\nei_2(v)=\nei_2(w)\right)\to 0.
}
In order for $\nei_2(v)=\nei_2(w)$, the degree of $v$ ($\deg(v)$) must be equal to that of $w$ and the degrees of the neighbors of $v$ and $w$ must be equal as multi-sets.
 Note that we can write $\deg(v)=B_v+I$ and $\deg(w)=B_w+I$ where $B_v$ and $B_w$ are independent with distribution $\Bi(N-2, p_N)$
 and $I$ is the indicator that $v$ and $w$ have an edge between them. We bound the chance that $v$ and $w$ have the same degree 
 and the chance of sharing too many neighbors as follows.

\begin{enumerate}
\item The Chernoff bound of Lemma~\ref{lemchernbd} applied to the binomial distribution implies that for all $0<\eps_1<1/2$,
\be{
\IP\left( \deg(v)\in Np_N(1\pm \eps_1)\right)\geq 1-2 \exp\big\{-\tsfrac{\eps_1^2}{3} N p_N\big\}.
}
\item Noting that the event $\deg(v)=\deg(w)$ is independent of $I$, the indicator that $v$ and $w$ have an edge between them,
we use  the local limit theorem for the binomial distribution to find for $C$ not depending on $N$,
\bes{
\IP(\deg(w)&=\deg(v)|\deg(v)\in Np_N(1\pm \eps_1))\leq \frac{C}{\sqrt{N p_N}}.
}
\item Denote the common degree of $v$ and $w$ by 
\be{
M:=\deg(v)\Ind \left[\deg(w)=\deg(v)\in Np_N(1\pm \eps_1))\right],
}
assuming the conditioning of the items above hold (and zero otherwise), and let $K=M-|\nei_1(v)\cap\nei_1(w)|-I$ be the number of 
neighbors of $v$ and $w$ that are connected to exactly one of $v$ or $w$. Given $M>0$, the neighbors of $v$ and $w$ are each chosen
uniformly from the $N-1$ possible neighbors. Thus if $v$ and $w$ are not neighbors,  then
$M-K$ is hypergeometric with $M$ draws, $M$ marked balls and $N-2$ total balls and if $v$ and $w$ are neighbors,
then $M-K$ is hypergeometric with $M-1$ draws, $M-1$ marked balls and $N-2$ total balls. In either case,
after noting that hypergeometric distributions can be represented as sums of \emph{independent} 
indicators \cite{Pitman1997} the Chernoff bound of Lemma~\ref{lemchernbd} implies that
for any $0<\eps_2<1/3$, 
\ba{
\IP\left( K\geq  (1-\eps_2) M \big| M, \{M>0\}\right)
	&=\IP\big( M-K\leq  \big(\eps_2\tsfrac{M}{\IE[M-K]}\big) \IE[M-K] \big|M, \{M>0\}\big) \\
	&\geq 1- \exp\left\{-\tsfrac{M}{3}\Big(\eps_2 - \tsfrac{M}{N-2}\Big)\right\},
}
where we are using that if $M>0$, then $M\in Np_N(1\pm \eps_1)$ and so $M/\IE[M-K] \approx N/M \gg 1$. 
Thus, if $M>0$ then
\ba{
\tsfrac{1}{3} Np_N < Np_N(1-\eps_2)(1-\eps_1)\leq K \leq M < 2Np_N.
}
\end{enumerate}

Given $M>0$ and $K$, let
\ba{
\{D(v)\}:=\{D_1(v),\ldots, D_K(v)\}, \,\, \mbox{ and } \,\, 
\{D(w)\}:=\{D_1(w),\ldots, D_K(w)\}
}
 denote the multi-set of 
 degrees of the $K$ non-intersecting
neighbors of $v$ and $w$, respectively. The three items above imply the following bound.
\ban{
\IP(\nei_2(v)=\nei_2(w)) 
& \leq 2 \exp\left\{-\tsfrac{\eps_1^2}{3} N p_N\right\}
	 +\frac{2 C \exp\left\{-\tsfrac{Np_N}{6} (\eps_2 - 5 p_N)\right\}}{\sqrt{N p_N}}   \label{der2p} \\
	 &\qquad + \frac{C}{\sqrt{N p_N}} \IP\left[ \{D(v)\}=\{D(w)\} |\mathcal{E}\right], \label{der3p}
}
where  $\mathcal{E}=\{M>0, (1/4)Np_N< K < 2 Np_N\}$. 
Since $\log(N)^2/N \leq p_N\leq N^{-3/5}$, the two terms of~\eq{der2p} are easily seen to be $\lito(1/N^2)$ so we only
need to bound~\eq{der3p}.

Write $D_i(v)=V_i+A_i+1$, where $V_i$ is the number of edges between the vertex $v_i$ representing $D_i(v)$ and the $N-2K-2$ vertices
not in $\{v\}\cup\{w\}\cup\left(\nei_1(v)\cup\nei_1(w)\right)/\left(\nei_1(v)\cap\nei_1(w)\right)$ and $A_i$ is the number of
edges between $v_i$ and the remaining $2K-1$ potential neighbors, not including $v$
and $w$. Similarly, write $D_i(w)=W_i+B_i$.
Note that given $M>0$ and $K$,  $\{V_1,\ldots, V_K\}$ and $\{W_1,\ldots, W_K\}$
are two independent (unordered) collections of i.i.d.\ $\Bi(N-2K-2, p_N)$ random variables that are also independent of 
the $A_i$'s and $B_i$'s.

We show $(i)$ that with sufficiently high probability, $A_i$ and $B_i$ are bounded by a constant and $(ii)$ that the chance that
independent binomial multi-sets are within constants is small.

For $(i)$,  the  $A_i$'s and $B_i$'s form the degrees of an \ER\ graph on $2K$ vertices with edge parameter $p_N$
and so are each marginally distributed $\Bi(2K-1, p_N)$. 
Thus,
\bes{
\IP(A_i < x, B_i &<x, i=1,\ldots, K)
\geq 1-2K \IP(A_1>x)
 \geq 1-2K \left(\frac{e}{x}\right)^x \left((2K-1)p_N\right)^x,
}
where we have used standard tail bounds on the binomial distribution in the Poisson regime stated in Lemma~\ref{lemchernbd}.
Note that setting $x=13$ (any $x>12$ works) and using that $p_N\leq N^{-3/5}$ we find that if $K < 2 Np_N$, then
\ben{
\IP(\max_{i}\{A_i, B_i\}>13) \leq \lito(N^{-2}).
} 

At this point we only need to show that for $\{V_1,\ldots, V_K\}$ and $\{W_1,\ldots, W_K\}$
two independent (unordered) collections of i.i.d.\ $\Bi(N-2K-2, p_N)$ random variables 
with $Np_N/3<K < 2 Np_N$, and for fixed non-negative $A_1, \ldots, A_K, B_1, \ldots, B_K$ 
such that each $A_i$ and $B_i$ are no greater than $13$, 
\besn{\label{tt19}
\IP(\{V_1+A_1,\ldots,V_K+A_K\}=\{W_1+B_1,\ldots,W_K+B_K\}) = \lito(1/N^2).
}

Rather than dealing with the multi-sets, we look instead at the (nearly multinomial) vectors of counts. For $i=0,\ldots, N-2K+12$, let 
$X_i=|\{j: V_j+A_j=i\}|$ be the number of the $(V_j+A_j)$'s that are 
equal to $i$ and $Y_i=|\{j: W_j+B_j=i\}|$ be the analogous counts for the $(W_j+B_j)$'s. 
The left hand side of~\eq{tt19} is bounded by
\besn{\label{tt100}
&\IP\left( X_j = Y_j, j=-13, \ldots,  N-2K+12\right)
 \leq \IP\left( X_{j_i}=Y_{j_i},  i=0, \ldots, \lfloor \alpha \sqrt{N p_N} \rfloor-1 \right),
}
where $\alpha>0$ will be chosen later and we define $j_i= \lfloor Np_N \rfloor +i$. To shorten formulas define the index set
$\Index=\Index(N, \alpha):=\{0,\ldots,  \lfloor \alpha \sqrt{N p_N} \rfloor-1 \}$.
We bound the probability~\eq{tt100} by showing first that for an appropriate $\delta>0$, 
\besn{\label{tt101}
&\IP(X_{j_i}> (1+\delta) \IE X_{j_i}, \mbox{ for some } i\in \Index 
) = \lito(N^{-2}),
}
and then that given $X_{j_i}\leq  (1+ \delta) \IE X_{j_i}$ for all $i\in \Index$, 
 we apply
the local central limit theorem to the $Y_{j_i}$ (represented as sums of independent Bernoulli variables) to show that
the event on the right hand side of~\eq{tt100} has chance $\lito(N^{-2})$.

To show~\eq{tt101}, first note that by the local central limit theorem for the binomial distribution (noting that $p_N\to0$),
 there are positive constants $c_1=c_1(\alpha)$ and $c_2$ such that for all $i\in \Index$ 
 and $k=1,\ldots, K$,
\besn{\label{tt102}
\frac{c_1}{\sqrt{N p_N}}\leq \IP(V_k+&A_k= j_i),  \IP(W_k+B_k= j_i)\leq \frac{c_2}{\sqrt{N p_N}}.
}
Thus, for each $j_i$, $X_{j_i}$ is a sum of 
$K$ independent Bernoulli variables, each having success probability upper and lower bounded
as per~\eq{tt102}, and using this, a union bound, Lemma~\ref{lemchernbd}, and the bounds on $K$ and $p_N$, we have
\ba{
\IP(X_{j_i}>(1+\delta) \IE X_{j_i}, &\mbox{ for some } i\in \Index) \\
&\qquad \leq \sum_{i\in \Index} 2 \exp\left\{- \frac{\delta^2}{2+\delta} \IE X_{j_i}\right\} \\
&\qquad \leq 2 \alpha \sqrt{N p_N} \exp\left\{- \frac{\delta^2}{2+\delta} \frac{ c_1}{3}\sqrt{N p_N }\right\} \\
&\qquad \leq 2 \alpha N^{1/5} N^{-\frac{\delta^2}{2+\delta} \frac{ c_1}{3}}\lito(1).
}
Now choosing 
\be{
\delta=\frac{1+\sqrt{1+40c_1/27}}{10c_1/27},
}
shows~\eq{tt101} is satisfied, since for this choice of $\delta$,
\be{
-\frac{\delta^2}{2+\delta} \frac{ c_1}{3}+1/5=-2.
}

To finish the proof, we show that for an appropriate choice of $\alpha$ (small),
\bes{
&\IP\left( X_{j_i}=Y_{j_i},  i\in \Index \big| X_{j_i} \leq (1+ \delta) \IE X_{j_i},  \mbox{ all } i\in \Index\right) =\lito(N^{-2}).
}
Let $K_0=K$ and $K_i=K-\sum_{\ell=0}^{i-1} Y_{j_\ell}$
and define $\mc{F}_i$ to be the sigma field generated by $Y_{j_0}, \ldots, Y_{j_i}$. Observe that 
for each $i\in \Index$,  given $\mc{F}_{i-1}$, $Y_{j_i}$ is a sum of $K_i$ Bernoulli variables, 
each having success probability $Q$ satisfying (using~\eq{tt102})
\bes{
\frac{c_1}{\sqrt{N p_N}}& \leq \frac{c_1/\sqrt{N p_N}}{1-i c_1/\sqrt{Np_N}}\leq Q 
\leq \frac{c_2/\sqrt{N p_N}}{1-i c_2/\sqrt{Np_N}}\leq \frac{c_2}{\sqrt{N p_N}}(1- \alpha c_2)^{-1}.
}
So we demand that $(1-\alpha c_2)>0$ which is not an issue: changing $\alpha$ affects only $c_1$ and $\delta$ in the argument above.
Moreover, by decreasing $\alpha$, we increase $c_1$, and as $\alpha\to0$, $c_1$ stays bounded from above (since it's no greater than $c_2$) and thus so does $\delta$. 
The local central limit for sums of independent Bernoulli variables 
implies that 
\besn{\label{tt1000}
\IP\big(Y_{j_i}=X_{j_i} \big| X_{j_i} \leq (1+ \delta) \IE X_{j_i} &\mbox{ all } i\in \Index; Y_{j_\ell}=X_{j_\ell} \mbox{ all } \ell=0,\ldots, i-1; \mc{F}_{i-1} \big) \\
&\leq C\left[ K_i \frac{c_1}{\sqrt{N p_N} }\left(1-\frac{c_2}{\sqrt{N p_N}}(1- \alpha c_2)^{-1}\right)\right]^{-1/2},
}
for some constant $C$. Now the condition that $X_{j_i} \leq (1+ \delta) \IE X_{j_i}$ and the lower bound on $K$ implies that 
\ba{
K_i&\geq \frac{N p_N}{3}- (1+\delta) \sum_{\ell=0}^{i-1} \IE X_{j_\ell}  \\
&\geq  \frac{N p_N}{3}- (1+\delta) \sum_{\ell\in \Index} \IE X_{j_\ell} \geq  \frac{N p_N}{3}- (1+\delta)\alpha  2 c_2 N p_N,
}
where we have used that $\IE X_{j_i}\leq K c_2/\sqrt{N p_N}\leq 2c_2 \sqrt{N p_N}$. By choosing $\alpha$ small enough (so that
$1/3-2c_2(1+\delta)\alpha >0$) we find that $K_i$ is at least of order $N p_N$ for all $i\in \Index$ so that~\eq{tt1000} is
$\bigo\left( (Np_N)^{-1/4} \right)$. Now moving through $\Index$ sequentially, we have  
\ba{
&\IP\left( X_{j_i}=Y_{j_i},  i\in \Index \big| X_{j_i} \leq (1+ \delta) \IE X_{j_i},  \mbox{ all } i\in \Index\right)\\
&\hspace{1cm}=\exp\left\{-\frac{\alpha}{4} \sqrt{N p_N} \left(\log(N p_N) + \bigo(1)\right) \right\} \\
&\hspace{1cm}\leq \exp\left\{-\frac{\alpha}{4} \log(N)  \left(\log(\log(N)) + \bigo(1)\right) \right\} =\lito(N^{-2}).
\qedhere
}
 
 \end{proof}

\begin{lemma}\label{lemchernbd}
Let $X$ be the sum of independent indicators. Then for any $\eps>0$,
\ba{
\IP(X \leq \IE X (1-\eps))&\leq \exp\left\{-\frac{\eps^2}{2} \IE X\right\}, \\
\IP(X \geq \IE X (1+\eps))&\leq \exp\left\{-\frac{\eps^2}{2+\eps} \IE X\right\}.
}

If $X$ is a binomial distribution and $x>0$, then
\be{
\IP(X>x)\leq \left(\frac{e}{x}\right)^x (\IE X)^x.
}
\end{lemma}
\begin{proof}
The first statement is a standard Chernoff bound for sums of independent indicators. The second follows
in the usual way but we prove this particular form. For any $\theta>0$, a direct computation yields
\be{
\IP(X>x)\leq e^{-\theta x} \IE e^{\theta X} \leq \exp\left\{\IE X (e^\theta - 1)-\theta x\right\}.
}
Setting $\theta =\log(1+x/\IE X)$ in the previous formula and simplifying yields
\be{
\IP(X>x)\leq \left(\frac{e}{x+\IE X}\right)^x (\IE X)^x \leq \left(\frac{e}{x}\right)^x (\IE X)^x,
}
as desired.
\end{proof}

We finish the section with a couple open problems. In Theorem~\ref{properdense} is it possible to identify from $2$-neighborhoods?
What happens in the regime of $p_N$ we don't handle, where $\omega(N^{-1})= p_N =\bigo(\log(N)^2/N)$? 

\section{The Random Jigsaw Puzzle}\label{secpuzz}

We use a more intuitive description than that of Example~\ref{ex3} in the introduction.
The puzzle is given by an $n \times n$ grid of squares where adjacent squares share an edge. 
Each {\em edge} of a square in the grid is colored uniformly at random from one of $q$ colors.  A {\em piece} of the puzzle consists of a 
``vertex" at the center of the square along with the four adjacent colored edges. 
Vertices at the border have a half-edge that also gets colored, so each vertex has exactly four edges associated to it and ``edge pieces" are not distinguished. We allow the pieces to be rotated, but not flipped over, and the
goal is to identify the original colors of edges of the puzzle, up to rotation (but not flips) 
of the puzzle.

We first use blocking configurations to obtain an easy negative result.
\begin{proposition}\label{propjig1}
If $q = o(n^{2/3}) $ then the probability of identification goes to $0$ as $n \to \infty$. 
\end{proposition}

\begin{proof}
Call a pair of pieces {\em aligned} if it is at position $(2j,2i),(2j,2i+1)$. 
Let $X_{i,j,i',j'}$ be the indicator of the following event. Consider the map $\pi: (x,y) \to (x-2j+2j',y-2i+2i')$. Let $X_{i,j,i',j'}$ be $1$ if all 
edges emanating from $(2j,2i),(2j,2i+1)$ have the same color as their $\pi$ images except that the edge connecting $(2j,2i)$ and $(2j,2i+1)$ 
has a different color than its image under $\pi$. Note that if $X_{i,j,i',j'} = 1$ then there isn't a unique solution to the puzzle as the two aligned parts can be exchanged. Note that here we use the fact that with high probability there are no automorphism of the labelled puzzle (even excluding two neighborhoods). 
Let 
\[
Y = \sum_{(i,j) \neq (i',j')} X_{i,j,i',j'}.
\]
Then $\IE X_{i,j,i',j'} = q^{-6}(1-1/q)$ and moreover, it is easy to check that the $X_{i,j,i',j'}$ are pairwise independent. 
Thus
\[
\var(Y) =  \sum_{(i,j) \neq (i',j')} \var(X_{i,j,i',j'}) \leq \IE [Y]
\]
It follows that if $n^4 q^{-6} \to \infty$ then $\IE[Y] \to \infty$ and so by the second moment method, $\IP[Y \geq 1] \to 1$, concluding the proof.
\end{proof} 

On the other hand, if $q \gg n^4$, then by considering expectations, the number of edges with the same color tends to zero in probability and identification is trivial,
so if $q=\omega(n^4)$, then the probability of identification tends to $1$ as $n\to\infty$.
In fact we can do better.
\begin{proposition}\label{propjig2}
If $q = \omega(n^2)$ then it is possible to assemble  the puzzle with probability tending to one. 
More formally, if $q = \omega(n^2)$ then there exists an algorithm such that the probability it
correctly assembles the puzzle tends to one.
\end{proposition}

\begin{proof} 
We show that with probability tending to one, we can assemble the puzzle by
first joining edges with colors that appear exactly once in the puzzle and then filling in any remaining holes. 
Write $q = 2 c n (n+1)$ and let $m = 2 n(n+1)$ be the number of edges.
Let $U$ be the number of colors which appear exactly once. 
Then for large enough $m$,
\[
\IE U = q \frac{m}{q} (1-1/q)^{m-1}\geq m(1-2/c).
\]
Also note that $U$ is a function of the independent edge colors such that if a single color changes, then $U$ can change by at most $2$.
Thus we can apply McDiarmid's inequality for bounded differences to obtain that
\be{
\IP(U\geq m(1-2/c))\geq 1- \exp\left\{\frac{-m}{2c^2}\right\}. 
}
Given $U$, the locations of the edges that receive unique colors is 
exchangeable and so on $U\geq m(1-2/c)$, $U$ dominates the Bernoulli-$(1-3/c)$ product measure
on edges with chance at least $1-\exp\{-m(1-2/c)^2/(3-1/c) \}$, using, e.g., the Chernoff bound of Lemma~\ref{lemchernbd}. 
Thus, on the good event that $U\geq m(1-2/c)$ and at most $m(1-2/c)$ of the Bernoulli variables are $1$,
we can generate the locations of the unique colored edges by first generating the Bernoulli variables on edges 
and then adding the appropriate number of unique colors to the remaining edges chosen uniformly at random. 

If $c$ is large enough so that $1-3/c > 0.9$ (say), then 
standard results in percolation theory 
\cite[(8.97-8)]{Grimmett1999}
imply that 
the 
graph induced by the positive Bernoulli variables in the box (which on the good event 
are dominated by the unique edge color indicators) 
has a connected component touching all boundaries.  
Once such a component is determined, it is not hard to complete the puzzle.
By considering expectations,
the probability of having two pieces that share two or more colors tends to zero. 
Thus given a location of a piece neighboring two pieces that are already assembled -- i.e., an empty corner -- there is a unique piece that can fit there. 

Consider the process of starting with component formed by joining edges with unique colors 
and then repeatedly adding pieces to vacant corners. 
With probability tending to one, when this process terminates, the collection of vertices covered has no empty corners.
It is easy to see that this implies that the complete puzzle has been recovered. 
\end{proof} 

\begin{remark}
We have assumed that ``edge" pieces of the puzzle cannot be distinguished from interior pieces. If the edge pieces can be distinguished,
then the proposition still holds since with probability tending to one, it is possible to construct the border by matching colors 
that only appear once on the border and then filling in the interior using corners as is done in the proof above. It's interesting
that without the border, we need a non-trivial result from percolation theory to start the algorithm.
\end{remark}

\section{Conclusion and Additional Open Problems}\label{secopen}

A number of open problems regarding sharper bounds and extensions to other models are mentioned in the text and can be summarized as follows:

\begin{problem} \label{p:1}
For the graph shotgun problem on boxes in $\bz_n^d$ with labels given by i.i.d., Ising, Potts model, proper coloring etc., find the threshold 
for the graph identification problem. 
\end{problem} 
It is natural to consider canonical fixed graphs other than the lattice. 
As illustrated in the introduction, the case of regular trees should be rather straightforward for many of these models. 
However, other families of graphs may be amenable to analysis, e.g., expander graphs.

\begin{problem} \label{p:2} 
For the graph shotgun problem on a random graph model, e.g., \ER, preferential attachment, configuration, random regular graphs, etc., find the threshold for the graph identification problem. 
\end{problem}
This question applies to both the labeled and unlabeled case.
We can ask about generalizations of the random jigsaw problem.
\begin{problem}
Find the threshold for identification for the jigsaw problem on other lattices, for example, hexagonal pieces or higher dimensional square lattices. (Thanks to a reviewer for the suggestion to add this problem.)
\end{problem}

Another problem is to study the setup with more realistic assumptions.
\begin{problem}
Analogous to DNA shotgun assembly, in practical problems the neighborhoods will be of different sizes and there will be errors in the samples. How does this affect identifiability? 
\end{problem}

It is also interesting to understand if the graph identification problem shares properties of other constraint satisfaction problems:
\begin{problem} \label{p:3}
Are there graph shotgun problems for which there is a ``computationally hard" but identifiable regime.
\end{problem}
This problem identifies graph shotgun assembly as a constraint satisfaction problem: for each neighborhood we have to find all intersecting 
neighborhoods. 
In the language of constraint satisfaction, the problem would be classified as {\em planted}, meaning that we start from a solution and then impose constraints based on the solution. 

\section*{Acknowledgments}
E.M.\ would like to acknowledge the support of the following grants: NSF grants DMS 1106999 and CCF 1320105, DOD ONR grant N00014-14-1-0823, and grant 328025 from the Simons Foundation.
N.R.\ received support from ARC grant DP150101459
and thanks Aslan Tchamkerten
for helpful discussions 
about DNA shotgun assembly.
We thank James Lee for suggesting the terminology ``jigs", Chenchao Chen for helpful comments, and reviewers.

\end{document}